\documentclass{amsart}
\usepackage{graphicx} % Required for inserting images
\usepackage{framed}
\usepackage{yufei}

\title{A Short Proof to Defant and Kravitz's theorem on the Length of Hitomezashi Loops}
\author{Qiuyu Ren}

\address{Department of Mathematics, University of California, Berkeley, Berkeley, CA 94720, USA}
\email{qiuyu\_ren@berkeley.edu}

\author{Shengtong Zhang}

\address{Department of Mathematics, Stanford University, Stanford, CA 94305, USA}
\email{stzh1555@stanford.edu}

\begin{document}

\maketitle
\begin{abstract}
    We provide a shorter proof to Defant and Kravitz's theorem (arXiv:2201.03461, Theorem 1.2) on the length of Hitomezashi loops modulo 8.
\end{abstract}
\section{Introduction}
Hitomezashi, a type of Japanese style embroidery, has recently attracted attention due to its interesting mathematical properties. We refer the reader to \cite{defant2022loops} and \cite{numberphile21} for great surveys on the topic.

Following Section 1 of \cite{defant2022loops}, we abstract the stitch patterns in a Hitomezashi artwork with graph-theoretic language.
\begin{definition}
    \label{defn:hitomezashi}
    Consider the graph $\Cloth_{\ZZ}$ on $\ZZ \times \ZZ$ with $(i, j)$ adjacent to $(i, j \pm 1)$ and $(i \pm 1, j)$. A \textbf{Hitomezashi pattern} is a subgraph of $\Cloth_{\ZZ}$ defined by two infinite sequences $\epsilon, \eta \in \{0, 1\}^{\ZZ}$, with edge set 
    $$\{\{(i, j), (i + 1, j): i \equiv \eta_j \bmod{2}\} \bigcup \{\{(i, j), (i, j + 1)\}: j \equiv \epsilon_i \bmod{2}\}.$$
    A \textbf{Hitomezashi loop} is a cycle in a Hitomezashi pattern. A \textbf{Hitomezashi path} is a path in a Hitomezashi pattern. In this paper, all cycles and paths have an orientation.
\end{definition}

In \cite{defant2022loops}, Defant and Kravitz obtained the following beautiful result about the length of the loop.
\begin{theorem}[\cite{defant2022loops}, Theorem 1.2]
\label{thm:loop-length}
Every Hitomezashi loop has length congruent to $4$ modulo $8$.
\end{theorem}
\begin{figure}[h]
    \centering
    \includegraphics[width=0.5\linewidth]{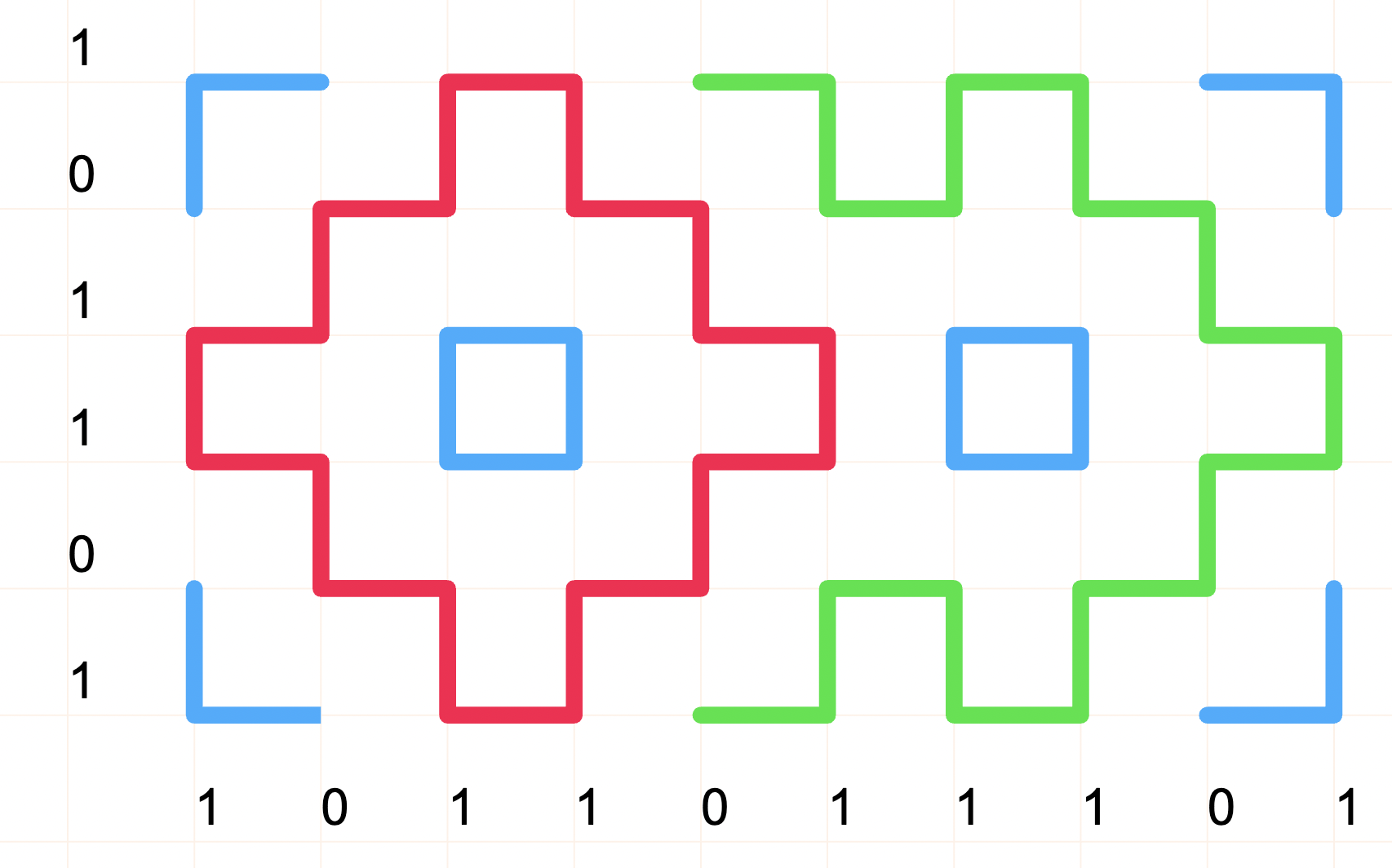}
    \caption{Part of a Hitomezashi pattern. The $\{0, 1\}$-labels denote values of $\epsilon$ and $\eta$ on each horizontal and vertical line. The red edges form a Hitomezashi loop, while the green edges form a Hitomezashi path. Note the red loop has length $20$, congruent to $4$ modulo $8$.}
    \label{fig:pattern}
\end{figure}
Their proof uses a complicated induction scheme, relying on additional structural results about the loop in \cite{Pete2008}. Given the simplicity of the theorem statement, we believe a shorter proof would be of interest.

In this paper, we present a two-page, self-contained proof of \cref{thm:loop-length}. Our main innovation is to induct on a different class of objects we call ``Hitomezashi excursions". This avoids many of the technical difficulties in \cite{defant2022loops}, which inducts on Hitomezashi loops.
\section{The proof}
We first collect some simple lemmata. 
\begin{lemma}
    \label{lem:parity-direction}
    On a Hitomezashi path, the edges starting at $(i, j)$ and $(i', j')$ are parallel if and only if $i + j \equiv i' + j' \bmod{2}$.
\end{lemma}
\begin{proof}
 Consecutive edges on the path are orthogonal and have different parity of $i + j$. 
\end{proof}
\begin{lemma}[\cite{defant2022loops}, Proposition 2.3]
\label{lem:vertical-direction}
Let $\cC$ be a Hitomezashi path. Then all edges of $\cC$ on a vertical line $x = a$ have the same direction, and the $y$-coordinates of their starting vertices have the same parity.
\end{lemma}
\begin{proof}
    Let $y_1$ and $y_2$ be the starting $y$-coordinates of two such edges. By \cref{lem:parity-direction}, we have $y_1 \equiv y_2 \bmod{2}$. The parities of $y_i$ determine the direction of the corresponding edges, so the lemma holds.
\end{proof}
We also need the following topological observation. Let $\cH_a$ denote the half plane $\cH_a := \{(x, y): x \geq a\}$.
\begin{lemma}
    \label{lem:disjoint-path}
    Suppose two continuous paths in $\cH_a$, one connecting $(a, y_1)$ and $(a, y_3)$ and the other connecting $(a, y_2)$ and $(a,y_4)$, are disjoint. Then we cannot have $y_1 \leq y_2 \leq y_3 \leq y_4$.\qed
\end{lemma}
We now introduce the ``excursion". See \cref{fig:excursions} for a visualization.
\begin{definition}
 Let $a$ be an integer. A \textbf{(Hitomezashi) $a$-excursion} is a Hitomezashi path with at least three vertices, whose start and end vertices lie on the vertical line $x = a - 1$, and all other vertices lie in $\cH_a$.
\end{definition}
Induction on excursions is much easier than induction on loops, as we illustrate in the next lemma.
\begin{lemma}
\label{lem:excursion-length}
The length of an $a$-excursion from $(a-1, i)$ to $(a-1, j)$ is congruent to $2\abs{j - i} + 1$ modulo $8$. 
\end{lemma}
\begin{proof}
Consider such an excursion $\cC$. Let $\abs{\cC}$ denote its length. We argue by induction on $\abs{\cC}$. When $\abs{\cC} = 3$, we must have $\abs{j - i} = 1$, so the result holds trivially.

Suppose the result holds for all excursions with smaller length than $\cC$. Reversing the orientation of $\cC$ and switch $i,j$ if necessary, assume $i<j$. Let $i = y_0, y_1, \cdots, y_t$ be the starting $y$-coordinate of the edges of $\cC$ lying on the vertical line $x=a$, in the order they appear when traversing $\cC$. By \cref{lem:vertical-direction}, we have
\begin{equation}
\label{eq:excursion-start-end-parity}
 i \equiv y_\ell \equiv j + 1 \bmod 2, \forall 0 \leq \ell \leq t.   
\end{equation}
and all edges of $\cC$ on the vertical line $x = a$ are in the same direction. We split into two cases.

\textbf{Case 1}: They point in the positive $y$-direction. In this case, for each $\ell \in [t - 2]$, the excursion $\cC$ contains disjoint paths $(a, i + 1) \to (a, y_{\ell})$, $(a, y_{\ell} + 1) \to (a, y_{\ell + 1})$, and $(a, y_{\ell + 1} + 1) \to (a, j)$, all lying $\cH_a$ and are disjoint. By \cref{lem:disjoint-path}, we must have
$$i < y_{\ell} < y_{\ell + 1} < j.$$
Thus we have
$$i = y_0 < y_1 < \cdots < y_{t} = j - 1.$$
\textbf{Case 2}: They point in the negative $y$-direction. In this case, let $s$ be the smallest integer in $[t]$ such that $y_{s - 1} < y_s$. Then clearly $y_{s - 1} < y_{s - 2} < \cdots < y_0.$ For any $i \in [t - 1]$ with $i > s$, $\cC$ contains disjoint paths $(a, y_{s - 1} - 1) \to (a, y_{s})$ and $(a, y_s - 1) \to (a, y_i)$ lying in $\cH_a$. By \cref{lem:disjoint-path}, we must have $y_{s - 1} < y_i < y_s$. Furthermore, $\cC$ contains disjoint paths $(a, y_s - 1) \to (a, y_{i})$ and path $(a, y_{i} - 1) \to (a, y_{i + 1})$ lying in $\cH_a$. Thus, we must have $y_{i + 1} < y_i$. To summarize, we have
$$y_{s - 1} < \cdots  < y_0 = i < y_{t} = j + 1 < y_{t - 1} < \cdots < y_s.$$
We partition $\cC$ by the $(t + 1)$ edges starting at $(a, y_i)(0 \leq i \leq t)$. Removing these edges, $\cC$ is partitioned into the edges $(a-1, i) \to (a, i)$, $(a, j) \to (a - 1, j)$, and $(a + 1)$-excursions $\cC_i(i \in [t])$. Furthermore, each $\cC_i$ has strictly smaller length than $\cC$, so they satisfy the induction hypothesis. We have
$$\abs{\cC} = 3 + t + \sum_{\ell = 1}^t \abs{\cC_\ell}.$$
We now consider the two cases above. In Case 1, $\cC_\ell$ is an excursion from $(a, y_{\ell - 1} + 1)$ to $(a, y_{\ell})$, so we have $\abs{\cC_{\ell}} \equiv 2\abs{y_{\ell} - y_{\ell - 1} - 1} + 1 \equiv 2(y_{\ell} - y_{\ell - 1}) - 1 \bmod{8}.$ Thus
$$\abs{\cC} \equiv 3 + t + \sum_{\ell = 1}^t (2(y_{\ell} - y_{\ell - 1}) - 1) \equiv 3 + 2(y_t - y_0) \equiv 2(j - i) +1 \bmod{8}.$$
In Case 2, $\cC_\ell$ is an $(a + 1)$-excursion from $(a, y_{\ell - 1} - 1)$ to $(a, y_{\ell})$, so for any $\ell \in [t]$,
$$\abs{\cC_{\ell}} \equiv 2\abs{y_{\ell} - y_{\ell - 1} + 1} + 1 \equiv \begin{cases}
2(y_{\ell} - y_{\ell - 1}) + 3, \ell = s \\
2(y_{\ell - 1} - y_{\ell}) - 1, \ell \neq s
\end{cases} \bmod{8}.$$
Thus
$$\abs{\cC} \equiv 3 + t + (2(y_{s} - y_{s - 1}) + 3) + \sum_{\ell \neq s, \ell \in [t]} (2(y_{\ell - 1} - y_{\ell}) - 1) \equiv 4y_s - 4y_{s - 1} - 2j + 2i + 5 \bmod{8}.$$
By \eqref{eq:excursion-start-end-parity}, we have $y_s - y_{s - 1} \equiv 0 \bmod{2}$ and $i - j \equiv 1 \bmod{2}$, so
$$\abs{\cC} \equiv 2(j - i) + 1 \bmod{8}.$$
In both cases, the induction hypothesis holds for $\cC$, so the induction step is complete.
\end{proof}
\begin{proof}[Proof of \cref{thm:loop-length}]
Consider a Hitomezashi Loop $\cL$. Let $a = \min_{(x, y) \in \cL} x$. By \cref{lem:vertical-direction}, we can orient $\cL$ such that all edges lying on $x = a$ point downward. Let $y_0, \cdots, y_{t - 1}$ be the $y$-coordinates of the starting vertices of edges of $\cL$ lying on the vertical line $x=a$, in the order they appear when traversing $\cL$, with $y_0$ being the largest among them. Define $y_{t} = y_0$.

For each $i \in [t - 2]$, $\cL$ contain disjoint paths $(a, y_0 - 1) \to (a, y_{i})$ and $(a, y_{i} - 1) \to (a, y_{i + 1})$ lying in $\cH_a$. Since $y_{i}, y_{i + 1}$ are less than $y_0$, we must have $y_{i + 1} < y_i$ by \cref{lem:disjoint-path}. Thus $y_0 > y_1 > \cdots > y_{t-1}$.

Removing the $t$ edges on $x = a$, $\cL$ is divided into $(a + 1)$-excursions $\cC_i$ from $(a, y_{i - 1} - 1)$ to $(a, y_{i})$ for $i \in [t]$. By \cref{lem:excursion-length}, we have
\begin{align*}
    \abs{\cL} &= t + \sum_{i = 1}^t \abs{\cC_i} \equiv t + \sum_{i = 1}^{t} (2\abs{y_{i - 1} - 1 - y_i} + 1) \\
    &\equiv t + \sum_{i = 1}^{t - 1} (2y_{i - 1} - 2y_i - 1) + (2y_{0} - 2y_{t - 1} + 3) \\
    &\equiv 4(y_0 - y_{t - 1}) + 4 \bmod{8}.
\end{align*}
By \cref{lem:parity-direction}, all $y_i$'s have the same parity. So $\abs{\cL}$ is congruent to $4$ modulo $8$.
\end{proof}

%\begin{proof}[Proof of \cref{thm:loop-length}, independent of "passage"]
%We slice the Hitomezashi loop along any vertical line not disjoint from it. Counterclockwisely along the loop, denote the $y$-coordinate of the intersections of the loop with the vertical line by $y_1,\cdots,y_{2t}$. Then by \cref{lem:excursion-length}, the length of the loop is $2\sum_{i=1}^{2t}|y_i-y_{i+1}|$, where $y_{2t+k}=y_k$. The points $y_i$'s are divided into three types: those that are local maxima: $y_i>y_{i-1},y_{i+1}$; those that are local minima; those that are neither. Relabel in order the local maxima by $y_1,\cdots,y_r$ and local minima by $z_1,\cdots,z_r$, and without loss of generality assume $y_1$ comes before $z_1$. Then the order of these points is exactly $y_1,z_1,y_2,z_2,\cdots,y_r,z_r$, and the length of the loop can be rewritten as $4\sum_{i=1}^r(y_i-z_i)$.

%Finally, notice each $y_i-y_{i+1}$ is odd, so $y_i-z_i$ is odd/even if and only if the loop points in different/same $x$-directions at the points $y_i,z_i$. By a topological consideration, the number of left-pointing $y_i$'s is one more than the left-pointing $z_i$'s. Thus there's an odd number of pairs $(y_i,z_i)$ pointing in different $x$-directions, this finishes the proof.
%\end{proof}

\begin{figure}[h]
    \centering
    \begin{minipage}{.475\textwidth}
      \centering
          \includegraphics[width=.5\linewidth]
          {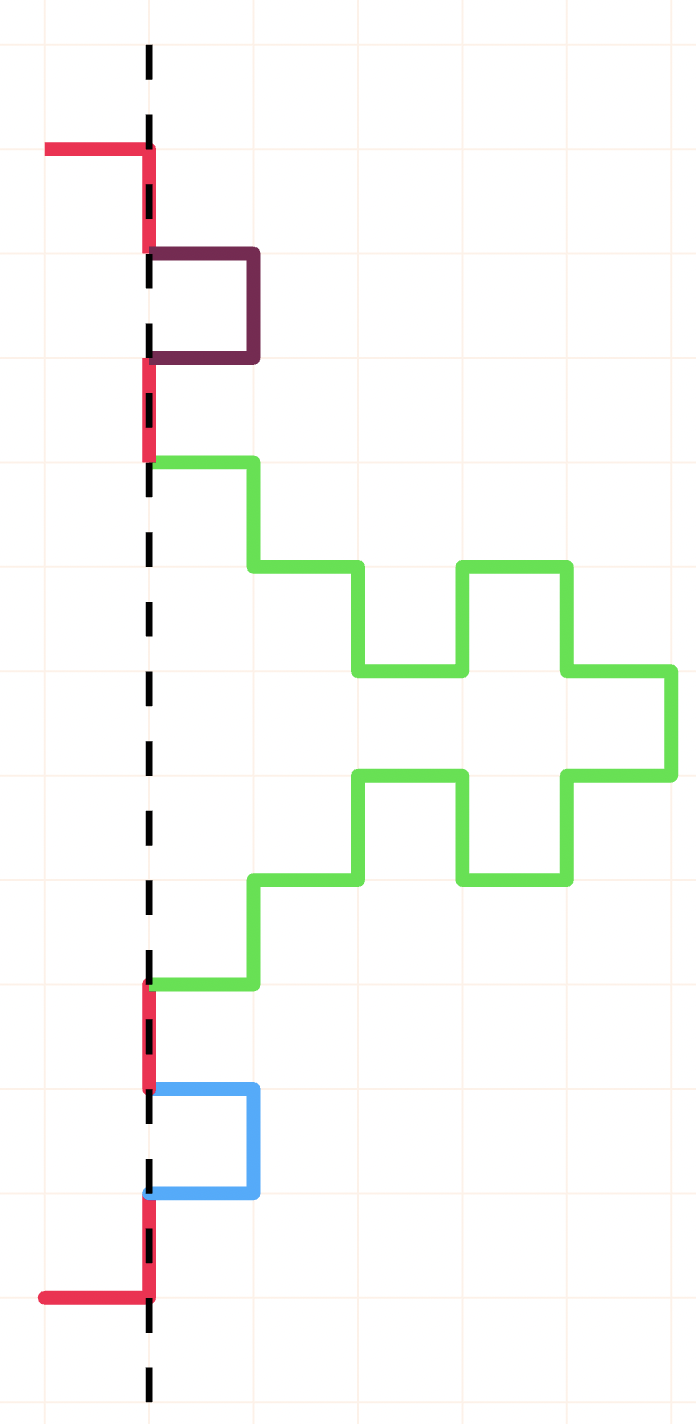}
    \end{minipage}%
    \begin{minipage}{.525\textwidth}
      \centering
      \includegraphics[width=.5\linewidth]{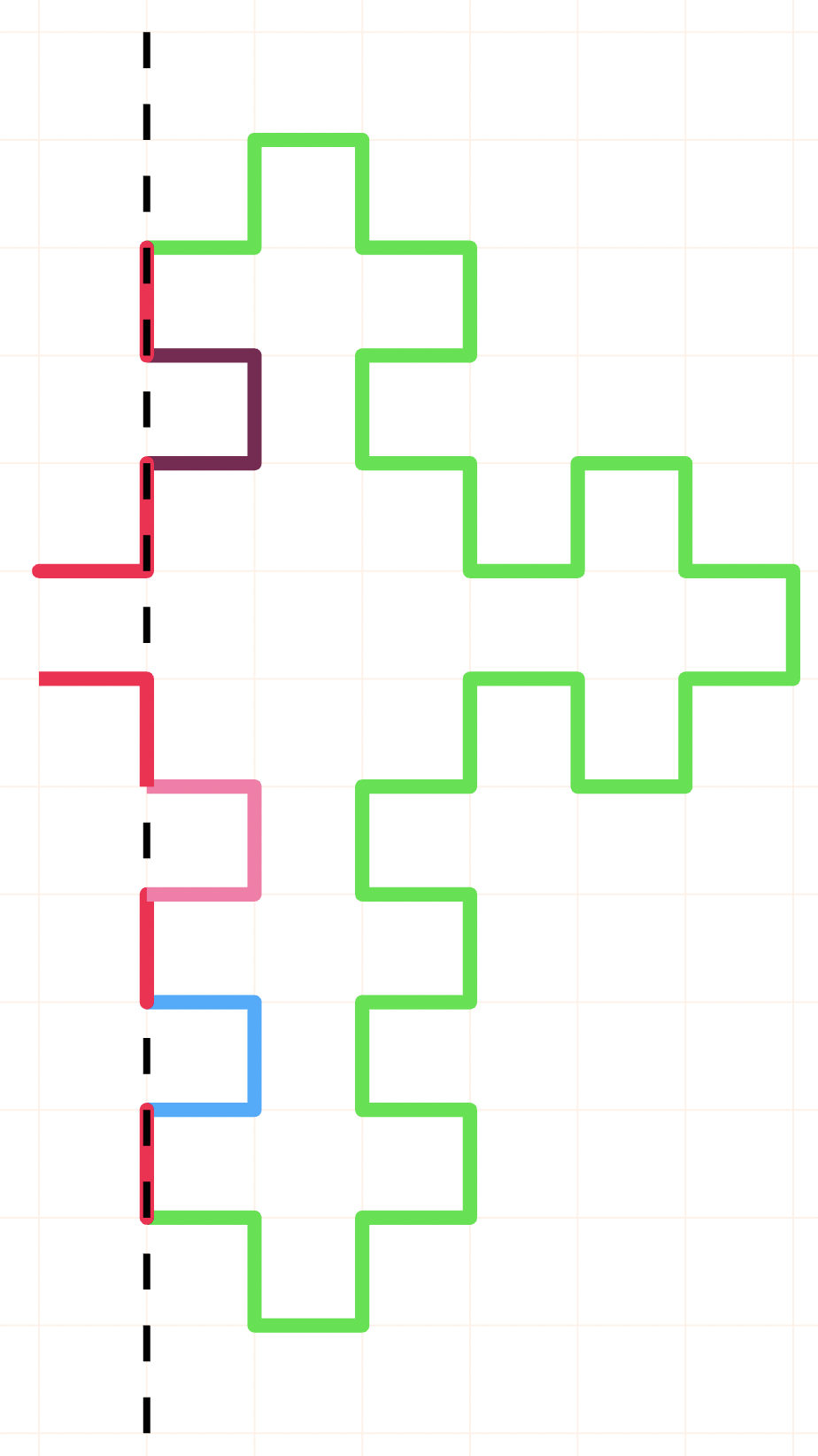}
    \end{minipage}
    \caption{Two $a$-excursions, corresponding to Case 1 and Case 2 in the proof. The dashed black line denotes the line $x = a$. The non-red subpaths are the $(a + 1)$-excursions we induct upon.}
    \label{fig:excursions}
\end{figure}

\section*{Acknowledgement}
Shengtong Zhang is supported by the Craig Franklin Fellowship in Mathematics at Stanford University.
\bibliographystyle{plain}
\bibliography{bib}

\end{document}